\RequirePackage{amsmath}
\documentclass{llncs}

\RequirePackage{etex} 
\usepackage{graphicx}
\usepackage[colorlinks,linktocpage]{hyperref}
\hypersetup{
   colorlinks   = true, 
   urlcolor     = blue, 
   linkcolor    = blue, 
   citecolor    = green 
}

\usepackage{microtype}

\usepackage[ruled,vlined,linesnumbered]{algorithm2e}
\usepackage{listings}
\usepackage{color}

\usepackage{float} 
\restylefloat{figure}

\usepackage{tikz}
\usetikzlibrary{positioning}
\usepackage{array}

\usepackage{pifont}
\usepackage{mathtools}
\usepackage{url}
\usepackage{multicol}
\usepackage{multirow}
\usepackage{systeme}
\usepackage{bbm}
\usepackage{minitoc}
\usepackage{bm}
\usepackage{amssymb}
\usepackage{graphicx}
\usepackage{multirow}
\usepackage{textcomp} 
\usepackage{enumerate}

\usepackage{fancyhdr}
\usepackage{esvect}

\usepackage[
	lambda,
	operators,
	advantage,
	sets,
	adversary,
	landau,
	probability,
	notions,	
	logic,
	ff,
	mm,
	primitives,
	events,
	complexity,
	asymptotics,
	keys]{cryptocode}

\usepackage{lineno}

\newcommand{\lb}{\mathsf{L}} 
\newcommand{\ub}{\mathsf{U}} 
\newcommand{\wbs}{\mathsf{w}} 
\newcommand{\std}{\mathsf{n}} 

\newcommand{\salsacore}{\mathsf{Salsa}}

\newcommand{\chachaperm}{\mathsf{ChaCha}_\pi}
\newcommand{\quarterround}{\mathcal{Q}}
\newcommand{\round}{\mathcal{R}}
\newcommand{\columnround}{\mathcal{R}^\mathsf{C}}
\newcommand{\diagonalround}{\mathcal{R}^\mathsf{D}}
\newcommand{\oracle}{\mathsf{Oracle}}

\newcommand{\qrprob}{\mathfrak{p}} 

\newcommand{\rr}{\mathsf{r}}

\newcommand{\rot}[2][]{\overleftarrow{#2}^{#1}}
\newcommand{\prot}[2][]{\overleftarrow{\overleftarrow{#2}}^{#1}} 

\begin{document}       
	
	\title{
		Rotational analysis of ChaCha permutation
	}
	\titlerunning{
	Rotational analysis of ChaCha permutation
	}
	\authorrunning{
	S.~Barbero et al.
	}
	\author{
	    Stefano Barbero \inst{1} \and
	    Emanuele Bellini\inst{2} \and
	    Rusydi Makarim \inst{2}
	}
	\institute{
		 Politecnico di Torino, Italy
		 \and
		 Cryptography Research Centre, Technology Innovation Institute, UAE
		}	
	\maketitle

\begin{abstract}
   We show that the underlying permutation of ChaCha20 stream cipher 
   does not behave as a random permutation for up to 17 rounds 
   with respect to rotational cryptanalysis. 
   In particular, we derive a lower and an upper bound 
   for the rotational probability 
   through ChaCha quarter round, 
   we show how to extend the bound to a full round and then 
   to the full permutation. 
   The obtained bounds show that 
   the probability to find what we call 
   a parallel rotational collision is, for example, 
   less than $2^{-488}$ for 17 rounds of ChaCha permutation, while
   for a random permutation of the same input size, 
   this probability is $2^{-511}$.
   We remark that our distinguisher is not 
   an attack to ChaCha20 stream cipher, 
   but rather a theoretical analysis of its internal permutation 
   from the point of view of rotational cryptanalysis.
\end{abstract}

\setcounter{tocdepth}{4}
\tableofcontents

\section{Introduction}
\label{sec:introduction}

Salsa20 \cite{bernstein2008salsa20} and ChaCha20 \cite{bernstein2008chacha} 
are two closely related stream ciphers developed by Daniel J. Bernstein. 
Salsa20, the original cipher, was designed in 2005, 
then later submitted 
to the eSTREAM project by Daniel J. Bernstein 
\cite{bernstein2005salsa20techreport}.
Its detailed specification can be found in 
\cite{bernstein2005salsa20techreport}.
ChaCha20 is a modification of Salsa20, 
published by Bernstein in 2008, 
aimed at increasing diffusion and performance on some architectures. 
Google has selected ChaCha20 along with 
Bernstein’s Poly1305 message authentication code 
as a replacement for RC4 in TLS, 
and its specifications can be found in \cite{DBLP:journals/rfc/rfc8439}.
Both ciphers are ARX (Add-Rotate-Xor) ciphers, i.e.  
built on a pseudorandom function 
based only on the following three operations:
32-bit modular addition, 
circular rotation, and
bitwise exclusive or (XOR).
This pseudorandom function is itself built 
upon a 512 bit permutation.
According to \cite{bernstein2015sphincs},
both permutations are not designed to simulate ideal permutations: 
they are designed to simulate ideal permutations with certain symmetries, 
i.e., 
ideal permutations of the orbits of the state space under these symmetries. 
The input of the Salsa and ChaCha function
is partially fixed to specific asymmetric constants, 
guaranteeing that different inputs lie in different orbits.
To our knowledge, while for Salsa 
some of these properties of "non-pseudorandomness" are well known, 
this is not the case for ChaCha (see Section \autoref{sec:related_works}).
Again to our knowledge, 
because of the use of asymmetric constants injected 
into the input state of the permutation, 
none of these properties can be used to attack the entire stream cipher, or 
other ciphers where these permutations have been reused, as 
Salsa20 permutation in the Rumba20 compression function
\cite{bernstein2007output}, 
a tweaked version of the ChaCha20 permutation 
in the BLAKE and BLAKE2 hash functions
\cite{aumasson2013blake2}, 
or ChaCha12 permutation in the original SPHINCS post-quantum signature scheme
\cite{bernstein2015sphincs} \footnote{The current SPHINCS submission to the NIST post-quantum standardization process does not use ChaCha anymore.}.

That said, studying mathematical properties of the Salsa and ChaCha permutations 
is still of theoretical interest, and 
it is useful to understand how these permutations 
can be reused to design other cryptographic primitives.

\subsection{Our contribution}
\label{sec:our_contribution}

In this work,
we show that ChaCha permutation does not behave as a random permutation, 
with respect to rotational cryptanalysis.
To do so, we first derive and formally prove 
a lower and an upper bound for the probability of 
the propagation of rotational pairs through ChaCha quarter round. 
We provide experimental evidence of the correctness of the bounds 
by testing them on a toy version of ChaCha permutation. 
We then show how to extend the bounds to a full round and then to the full permutation. 
The obtained bounds allow us to distinguish
ChaCha permutation, with for example 17 rounds, from a random permutation 
by using $2^{489}$ calls to an oracle 
running either ChaCha permutation or the random permutation. 
To do so, we prove that what we call 
a \emph{parallel rotational collision}, 
is more likely to happen in ChaCha permutation, 
rather than in a random permutation.
For example, 
such a collision happens with probability less than $2^{-488}$ 
for ChaCha permutation with 17 rounds,
while with probability $2^{-511}$ for a random permutation.
This distinguisher is \emph{not} an attack to ChaCha20 stream cipher, 
but rather a theoretical analysis of its permutation from the point of view of rotational cryptanalysis.

\subsection{Outline of the paper}
\label{sec:outline}

In \autoref{sec:related_works} we briefly summarize the existing studies on the core function of Salsa and ChaCha stream ciphers.
In \autoref{sec:chacha_permutation_description}, 
we introduce the notation used throughout this manuscript and recall ChaCha permutation specifications.
In \autoref{sec:rotational_difference_propagation}, 
we derive the lower and upper bound on the probability of 
the propagation of a rotational pair for ChaCha quarter round, 
for the full rounds, and for the full permutation.
In \autoref{sec:distinguisher}, 
we describe a distinguihser exploiting the above mentioned bounds.
Finally, in \autoref{sec:conclusion}, we conclude the manuscript.

\subsection{Related works}
\label{sec:related_works}

Often, rather than only considering the underlying permutation of Salsa and ChaCha, 
researchers study the so called Salsa (or ChaCha) \emph{core} function 
(also called ChaCha \emph{block function} in \cite{DBLP:journals/rfc/rfc8439}), 
whose output consists in applying the permutation and then 
xoring the output of the permutation with its input.

Already in the specifications of Salsa20 \cite{bernstein2005salsa20}, 
there is an example showing how the 0 vector is a fixed point for Salsa core function. This is also true for ChaCha.

In \cite{hernandez2008salsa20},
the authors find\footnote{According to the authors and to  \cite{bernsteinresponse}, 
most of these results were already informally observed by
Matt Robshaw in June 2005, and independently posted to sci.crypt by
David Wagner in September 2005, but we could not find any reference besides \cite{googlegroupssalsa20}.}
an invariant for Salsa core main building block, 
the quarterround function, 
that is then extended to the row-round and column-round functions. 
This allows them to find an input subset of size $2^{32}$ 
for which the Salsa20 core behaves exactly as the transformation $f(x)=2x$.
This allows to 
construct $2^{31}$ collisions for any number of rounds. 
They also show a differential characteristic with probability one 
that proves that the Salsa20 core does not have 2nd preimage resistance.
In \cite{bernsteinresponse}, 
it is pointed out that none of the results in \cite{hernandez2008salsa20}
has an impact on security of Salsa20 stream cipher, 
due to the use of fixed constants in the input.
Indeed, 
Salsa20 is not designed to be a collision-resistant compression function
\cite{salsacorewebsite}.

In the Salsa20 security document \cite[Section 4]{bernstein2005salsa20sectechreport}, 
two other symmetries of the cipher are reported, i.e. 
\begin{itemize}
    \item shifting the entire Salsa20 core input
array along the diagonal has exactly the same effect on the output, i.e. 
\begin{align*}
    \begin{bmatrix}
    y_{0,0} & y_{0,1} & y_{0,2} & y_{0,3} \\
    y_{1,0} & y_{1,1} & y_{1,2} & y_{1,3} \\
    y_{2,0} & y_{2,1} & y_{2,2} & y_{2,3} \\
    y_{3,0} & y_{3,1} & y_{3,2} & y_{3,3} \\
    \end{bmatrix}
    &=
    \salsacore
    \left(
    \begin{bmatrix}
    x_{0,0} & x_{0,1} & x_{0,2} & x_{0,3} \\
    x_{1,0} & x_{1,1} & x_{1,2} & x_{1,3} \\
    x_{2,0} & x_{2,1} & x_{2,2} & x_{2,3} \\
    x_{3,0} & x_{3,1} & x_{3,2} & x_{3,3} \\
    \end{bmatrix}
    \right) 
    \\
    \begin{bmatrix}
    y_{3,3} & y_{3,0} & y_{3,1} & y_{3,2} \\
    y_{0,3} & y_{0,0} & y_{0,1} & y_{0,2} \\
    y_{1,3} & y_{1,0} & y_{1,1} & y_{1,2} \\
    y_{2,3} & y_{2,0} & y_{2,1} & y_{2,2} \\
    \end{bmatrix}
    &=
    \salsacore
    \left(
    \begin{bmatrix}
    x_{3,3} & x_{3,0} & x_{3,1} & x_{3,2} \\
    x_{0,3} & x_{0,0} & x_{0,1} & x_{0,2} \\
    x_{1,3} & x_{1,0} & x_{1,1} & x_{1,2} \\
    x_{2,3} & x_{2,0} & x_{2,1} & x_{2,2} \\
    \end{bmatrix}
    \right) \,;
\end{align*}
    \item the Salsa20 core operations are 
          almost compatible with rotation of 
          each input word by, say, 10 bits.
\end{itemize}
This shift and rotation structures
are eliminated by the use of fixed constants in the input diagonal. 
Precisely, the input diagonal is different from 
all its nontrivial shifts and 
all its nontrivial rotations and 
all nontrivial shifts of its nontrivial rotations. 
In other words,
two distinct arrays with this diagonal 
are always in distinct orbits under the shift/rotate group.

We are not aware of similar properties for the case of the ChaCha permutation.
In particular, we are not aware of any study of the rotational properties of the ChaCha permutation.

\section{ChaCha permutation description}
\label{sec:chacha_permutation_description}

In this section, 
we first define our notation, then 
we describe the specifications of ChaCha permutation.
We do not describe the entire details of ChaCha as a stream cipher. 

\subsection{Notation}
\label{sec:notation}

Let $\FF_2$ be the binary field with two elements, 
and $\mathcal{M}_{\std \times \std}(\FF_2^\wbs)$ the set of all $\std \times \std$ matrices with elements in $\FF_2^\wbs$.
We indicate 
with lowercase letters $\wbs$-bit words, i.e.
$x \in \FF_2^\wbs$,  
with bold lower case letters vectors of $\std$ words, i.e.
$\bm{x} \in (\FF_2^\wbs)^\std$, and 
with uppercase letters a $\std \times \std$ matrix of $\std^2$ words, i.e. 
$X \in \mathcal{M}_{\std \times \std}(\FF_2^\wbs)$.

We use the following notation:
\begin{itemize}
\item $\oplus$ for the bitwise exclusive or (XOR), 
      i.e. the addition in $\FF_2^\wbs$; 
\item $\boxplus$ for the $\wbs$-bit addition $\mod 2^{\wbs}$;
\item $\boxplus_{i=1}^{k}a_{i}$ for the $\wbs$-bit addition $\mod 2^{\wbs}$ of $k$   words $a_{1}, \ldots,a_{k}$
\item $\lll \rr$ and $\ggg \rr$ 
      for constant-distance left and right, respectively, 
      circular rotation of $\rr$ bits of a $\wbs$-bit word 
      (with $\wbs > \rr$). 
      When needed, we also use the following more compact notation:
      \begin{itemize}
        \item $\rot[\rr]{x} = x \lll \rr$;
        \item $\prot[\rr]{\bm{x}} = (\rot[\rr]{x_0}, \dots, \rot[\rr]{x_{\std-1}})$ 
              the parallel left circular rotation of a 
              $\std$-word vector 
        \item 
            $
            \prot[\rr]{X} = 
            \begin{bmatrix}
            \rot[\rr]{x_{0,0}} & \ldots & \rot[\rr]{x_{0,\std-1}} \\
            \vdots & \ddots & \vdots \\
            \rot[\rr]{x_{\std-1,0}} & \ldots & \rot[\rr]{x_{\std-1,\std-1}} 
            \end{bmatrix}
            $ 
            the parallel left circular rotation of 
            the $\wbs$-bit elements of the matrix 
            $
            X \in \mathcal{M}_{\std \times \std}(\FF_2^\wbs)
            $.
    \end{itemize}
    When clear from the context, 
    we omit the subscript $\rr$, and simply write
    $\rot{x}$, $\prot{\bm{x}}$, and $\prot{X}$.
\end{itemize}
In the case of ChaCha, 
we have 
$\std = 4$ and 
$\wbs = 32$.

\subsection{ChaCha permutation specification}
\label{sec:chacha_permutation_specs}

ChaCha permutation has a state of 512 bits, which 
can be seen as a $4 \times 4$ matrix whose elements are 
binary vectors of $\wbs = 32$ bits, i.e.
\begin{align*}
    X 
    =
    \{x_{i,j}\}_{\small{\substack{i=0,\ldots,3\\j=0,\ldots,3}}}
    = 
    \begin{bmatrix}
    x_{0,0} & x_{0,1} & x_{0,2} & x_{0,3} \\
    x_{1,0} & x_{1,1} & x_{1,2} & x_{1,3} \\
    x_{2,0} & x_{2,1} & x_{2,2} & x_{2,3} \\
    x_{3,0} & x_{3,1} & x_{3,2} & x_{3,3} \\
    \end{bmatrix}
    \in 
    \mathcal{M}_{\std \times \std}(\FF_2^\wbs)
    \,.
\end{align*}

\begin{definition}[ChaCha quarter round]
Let $x_i, y_i, i=0,1,2,3$ be $\wbs$-bit words, and
let $(y_{0},y_{1},y_{2},y_{3})
= \quarterround(x_0,x_1,x_2,x_3)$, where 
$\quarterround$ is ChaCha quarter round, defined as follows:

\begin{align}
\label{b0} b_{0}&=x_{0}\boxplus x_{1}\\
\label{b3} b_{3}&=\left(b_{0}\oplus x_{3}\right)\lll r_{1}\\
\label{b2} b_{2}&=b_{3}\boxplus x_{2}\\
\label{b1} b_{1}&=\left(b_{2}\oplus x_{1}\right)\lll r_{2}.
\end{align}
and
\begin{align}
\label{y0} y_{0}&=b_{0}\boxplus b_{1}\\
\label{y3} y_{3}&=\left(y_{0}\oplus b_{3}\right)\lll r_{3}\\
\label{y2} y_{2}&=y_{3}\boxplus b_{2}\\
\label{y1} y_{1}&=\left(y_{2}\oplus b_{1}\right)\lll r_{4}
\end{align}
\end{definition}


\begin{figure}
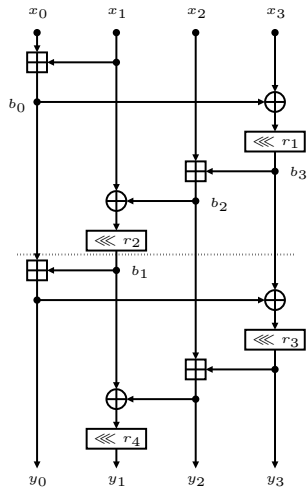

\centering
    \include{chacha_qr}
\caption{The ChaCha quarter round.}
	\label{fig:chachaquarterround}
\end{figure}

We show in Fig. \ref{fig:chachaquarterround} a schematic drawing of the Chacha quarter round. 
The permutation used in ChaCha20 stream cipher performs 20 rounds or, equivalently, 10 \emph{double rounds}.
Two consecutive rounds (or a \emph{double round}) of ChaCha permutation consist in 
applying the quarter round 
four times in parallel to the columns of the state 
(first round), and then
four times in parallel to the diagonals of the state
(second round). More formally:

\begin{definition}[ChaCha column/diagonal round]
Let 
$
X = 
\{x_{i,j}\}_{\small{\substack{i=0,\ldots,3\\j=0,\ldots,3}}}
$
and 
$
Y = 
\{y_{i,j}\}_{\small{\substack{i=0,\ldots,3\\j=0,\ldots,3}}}
$ 
be two matrices in 
$\mathcal{M}_{\std \times \std}(\FF_2^\wbs)$.

A \emph{column round} 
$Y = \columnround(X)$ is defined as follows, 
with $i = 0,1,2,3$:
\begin{align*}
(y_{0,i},y_{1,i},y_{2,i},y_{3,i})
= 
\quarterround(x_{0,i},x_{1,i},x_{2,i},x_{3,i})
\,.
\end{align*}

A \emph{diagonal round} 
$Y = \diagonalround(X)$ is defined 
as follows, 
for $i = 0,1,2,3$ and where each pedix 
is computed modulo $\std=4$:
\begin{align*}
(y_{0,i},y_{1,i+1},y_{2,i+2},y_{3,i+3})
= 
\quarterround(x_{0,i},x_{1,i+1},x_{2,i+2},x_{3,i+3}) 
\,.
\end{align*}
\end{definition}




\section{Propagation of rotational pairs}
\label{sec:rotational_difference_propagation}
In this section, 
we first define a set of necessary and sufficient conditions 
for the propagation of rotational pairs through ChaCha quarter round.
We then use these conditions to derive a lower and an upper bound for the probability of this propagation to happen through the quarter round.
Then, we describe how to extend the bounds to a full round, and 
finally to the full permutation.

\subsection{Conditions for rotational pairs propagation}
\label{sec:rotational_difference_propagation_conditions}

We are interested in studying the probability of 
the propagation through the quarter rounds of
\emph{rotational} pairs, i.e., of 
\begin{equation}
\label{prob}
\qrprob
=
\Pr[ 
(\rot[\rr]{y_{0}}, \rot[\rr]{y_{1}}, \rot[\rr]{y_{2}}, \rot[\rr]{y_{3}})
= 
\quarterround(\rot[\rr]{x_0}, \rot[\rr]{x_1}, \rot[\rr]{x_2}, \rot[\rr]{x_3})
]    
\,.
\end{equation}
To do so, we first prove the following proposition.
\begin{proposition}
Given ChaCha quarter round $\quarterround$ defined as above with  
the non negative integers $r_1, r_2, r_3, r_4 \le \wbs-1$, and 
given the rotational amount $\rr \le \wbs-1$, 
then
\begin{align*}
    (y_{0}\lll \rr,y_{1}\lll \rr,y_{2}\lll \rr,y_{3}\lll \rr)
    & = 
    \quarterround(x_0 \lll \rr, x_1 \lll \rr, x_2 \lll \rr, x_3 \lll \rr)
    \\
    & \iff
    \\
    \left(x_{0}\lll \rr \right)\boxplus\left(x_{1}\lll \rr\right)
    & =
    \left(x_{0}\boxplus x_{1}\right)\lll \rr
    \\
    \left(b_{3}\lll \rr \right)\boxplus\left(x_{2}\lll \rr \right)
    & =
    \left(b_{3}\boxplus x_{2}\right)\lll \rr
    \\
    \left(\left(x_{0}\boxplus x_{1}\right)\lll \rr\right)\boxplus\left(b_{1}\lll \rr\right)
    & = 
    \left(x_{0}\boxplus x_{1}\boxplus b_{1}\right)\lll \rr
    \\
    \left(y_{3}\lll \rr \right)\boxplus\left(\left(b_{3}\boxplus x_{2}\right)\lll \rr\right)&
    =
    \left(y_{3}\boxplus b_{3}\boxplus x_{2}\right)\lll \rr
    \,.    
\end{align*}
\end{proposition}

\begin{proof}
Let us consider what happens to the output if, 
instead of the input $(x_{0},x_{1},x_{2},x_{3})$, 
we use the input $(x_{0}\lll \rr,x_{1}\lll \rr,x_{2}\lll \rr,x_{3}\lll \rr)$, where every string is rotated $\rr$ places to the left.
First of all we find 
\begin{align}
\label{b0r}\widetilde{b_{0}}&=\left(x_{0}\lll \rr\right)\boxplus\left(x_{1}\lll \rr\right)\\
\label{b3r}\widetilde{b_{3}}&=\left(\widetilde{b_{0}}\oplus\left(x_{3}\lll \rr\right)\right)\lll r_{1}\\
\label{b2r}\widetilde{b_{2}}&=\widetilde{b_{3}}\boxplus\left(x_{2}\lll \rr\right)\\
\label{b1r}\widetilde{b_{1}}&=\left(\widetilde{b_{2}}\oplus\left(x_{1}\lll \rr\right)\right)\lll r_{2}
\end{align}
and 
\begin{align}
\label{y0r}	\widetilde{y_{0}}&=\widetilde{b_{0}}\boxplus\widetilde{b_{1}}\\
\label{y3r}	\widetilde{y_{3}}&=\left(\widetilde{y_{0}}\oplus\widetilde{b_{3}}\right)\lll r_{3}\\
\label{y2r}	\widetilde{y_{2}}&=\widetilde{y_{3}}\boxplus\widetilde{b_{2}}\\
\label{y1r} \widetilde{y_{1}}&=\left(\widetilde{y_{2}}\oplus\widetilde{b_{1}}\right)\lll r_{4}
\end{align}

Now, the conditions that must be simultaneously fulfilled in order to obtain
\begin{equation}\label{eqoutput}
(y_{0}\lll \rr,y_{1}\lll \rr,y_{2}\lll \rr,y_{3}\lll \rr)
=
(\widetilde{y_{0}},\widetilde{y_{1}},\widetilde{y_{2}},\widetilde{y_{3}}) 
\,,
\end{equation}
are the following:
\begin{align}
\label{y0=y0r}\widetilde{y_{0}}=y_{0}\lll \rr&\Longleftrightarrow\widetilde{b_{0}}\boxplus\widetilde{b_{1}}=\left(b_{0}\boxplus b_{1}\right)\lll \rr\\
\label{y3=y3r}\widetilde{y_{3}}=y_{3}\lll \rr&\Longleftrightarrow\left(\widetilde{y_{0}}\oplus\widetilde{b_{3}}\right)\lll r_{3}=\left(\left(y_{0}\oplus b_{3}\right)\lll r_{3}\right)\lll \rr\\
\label{y2=y2r}\widetilde{y_{2}}=y_{2}\lll \rr&\Longleftrightarrow\widetilde{y_{3}}\boxplus\widetilde{b_{2}}=\left(y_{3}\boxplus b_{2}\right)\lll \rr\\
\label{y1=y1r}\widetilde{y_{1}}=y_{1}\lll \rr&\Longleftrightarrow\left(\widetilde{y_{2}}\oplus\widetilde{b_{1}}\right)\lll r_{4}=\left(\left(y_{2}\oplus b_{1}\right)\lll r_{4}\right)\lll \rr
\end{align}
These constraints can be simplified, 
observing that from (\ref{y3=y3r}), 
considering the condition $\widetilde{y_{0}}=y_{0}\lll \rr$ and 
thanks to the distributive property of bit rotation 
with respect to $\oplus$, we have
\begin{align*}
\left(
\left(y_{0}\lll \rr
\right)
\oplus
\widetilde{b_{3}}\right)\lll r_{3} 
& = \left(\left(y_{0}\oplus b_{3}\right)\lll r_{3}\right)\lll \rr \\
& = \left(\left(y_{0}\lll \rr\right)\oplus\left(b_{3}\lll \rr\right)\right)\lll r_{3}
\,.
\end{align*}
Thus, we must have 
\begin{equation}
\label{b3=b3r}
\widetilde{b_{3}}=b_3\lll \rr
\end{equation}
and from (\ref{y1=y1r}) in an analogous way, using the condition $\widetilde{y_{2}}=y_{2}\lll \rr$, we find that
\begin{equation}\label{b1=b1r}\widetilde{b_{1}}=b_1\lll \rr\end{equation}
must hold.
Now considering (\ref{b3=b3r}) and equalities (\ref{b3}) and (\ref{b3r}) we easily observe that 
$$\left(\widetilde{b_{0}}\lll r_{1}\right)\oplus\left(\left(x_{3}\lll \rr\right)\ll r_{1}\right)=\left(\left(b_{0}\lll \rr\right)\lll r_{1}\right)\oplus\left(\left(x_{3}\lll \rr\right)\lll r_{1}\right)$$
and we find 
\begin{equation}\label{b0r=b0}
\widetilde{b_{0}}=b_0\lll \rr.
\end{equation}
In a similar way, considering (\ref{b1=b1r}) and equalities (\ref{b1}) and (\ref{b1r})  we have
$$\left(\widetilde{b_{2}}\lll r_{2}\right)\oplus\left(\left(x_{1}\lll \rr\right)\ll r_{2}\right)=\left(\left(b_{2}\lll \rr\right)\lll r_{2}\right)\oplus\left(\left(x_{1}\lll \rr\right)\lll r_{2}\right)$$
obtaining
\begin{equation}\label{b2r=b2}
\widetilde{b_{2}}=b_2\lll \rr.
\end{equation}

Thus condition (\ref{eqoutput}) corresponds to the following four conditions 
\begin{align}
\widetilde{b_{0}}&=b_{0}\lll \rr\\
\widetilde{b_{2}}&=b_{2}\lll \rr\\
\left(b_{0}\lll \rr\right)\boxplus\left(b_{1}\lll \rr\right)&=\left(b_{0}\boxplus b_{1}\right)\lll \rr\\
\left(y_{3}\lll \rr\right)\boxplus\left(b_{2}\lll \rr\right)&=\left(y_{3}\boxplus b_{2}\right)\lll \rr
\end{align}

or equivalently

\begin{align}
\label{cond1}\left(x_{0}\lll \rr\right)\boxplus\left(x_{1}\lll \rr\right)&=\left(x_{0}\boxplus x_{1}\right)\lll \rr\\
\label{cond2}\left(b_{3}\lll \rr\right)\boxplus\left(x_{2}\lll \rr\right)&=\left(b_{3}\boxplus x_{2}\right)\lll \rr\\
\label{cond3}\left(\left(x_{0}\boxplus x_{1}\right)\lll \rr\right)\boxplus\left(b_{1}\lll \rr\right)&=\left(x_{0}\boxplus x_{1}\boxplus b_{1}\right)\lll \rr\\
\label{cond4}\left(y_{3}\lll \rr\right)\boxplus\left(\left(b_{3}\boxplus x_{2}\right)\lll \rr\right)&=\left(y_{3}\boxplus b_{3}\boxplus x_{2}\right)\lll \rr
\end{align}
\qed
\end{proof}

\begin{remark}
Before trying to estimate the probability 
$\qrprob$, i.e., 
that all conditions (\ref{cond1}), (\ref{cond2}), (\ref{cond3}) and (\ref{cond4}) simultaneously hold,
we observe that 
the  rotation $r_{4}$ used in ChaCha quarter round function
is \emph{not} involved in \emph{any} of these equations, 
neither implicitly nor explicitly.
\end{remark}

\subsection{Bounds for the quarter round}
\label{sec:upper_lower_bounds_quarter_round}

We recall the result obtained in Corollary 4.12  by Daum 
\cite{DBLP:phd/de/Daum2005} 
on the propagation of 
the rotational probability 
with respect to modular addition.
\begin{proposition}
Let $a$ and $b$ be 
independent and uniformly distributed strings of $\wbs$ bits, and
$1\leq\rr \leq \wbs-1$ an integer.
Then
\begin{equation}
\begin{split}
\label{daum}
D 
&=
\Pr[
(a\lll \rr)\boxplus (b\lll \rr)
=
(a \boxplus b)\lll \rr
]
= \\
&=\frac{1+2^{-(\wbs-\rr)}+2^{-\rr}+2^{-\wbs}}{4}=
\frac{(2^{\rr}+1)(2^{\wbs-\rr}+1)}{2^{\wbs+2}}.
\end{split}
\end{equation}
\end{proposition}
The previous result can be generalized for the case where we have more than 2 addends.
\begin{proposition}\label{sumdaumgen}
	Let $a_1,a_2,\ldots,a_k$ be 
	independent and uniformly distributed strings of $\wbs$ bits, and
	$1\leq \rr\leq \wbs-1$ an integer.
	Then 
	\begin{equation}\label{daumgen}
	\Pr
	\left[
	\boxplus_{i=1}^{k}\left(a_{i}\lll \rr\right)=\left(\boxplus_{i=1}^{k}a_{i}\right)\lll \rr
    \right]
    = 
    \frac{F(\rr,k,\wbs)F(\wbs-\rr,k,\wbs)}{2^{k\wbs}}
    \end{equation}
where, for $1\leq q\leq \wbs-1$, 

\begin{equation}\label{fa}
\scalebox{0.93}{$
F(q,k,\wbs) 
= \sum_{h=0}^{\left\lfloor \frac{k\left(2^{q}-1\right)}{2^{\wbs}}\right\rfloor }\sum_{j=0}^{k}(-1)^{j}\binom{k}{j}\left(
\binom{h2^{\wbs}-(j-1)2^{q}-1+k}{h2^{\wbs}-(j-1)2^{q}-1}-\binom{h2^{\wbs}-j2^{q}-1+k}{h2^{\wbs}-j2^{q}-1}\right)
\,.$}
\end{equation}

\end{proposition}

\begin{proof}
	 In order to evaluate the number of solutions to 
	 \begin{equation}\label{summodai}
	   	\boxplus_{i=1}^{k}\left(a_{i}\lll \rr\right)=\left(\boxplus_{i=1}^{k}a_{i}\right)\lll \rr 
	 \end{equation}
i.e., how many $\wbs$-bit words $a_{1},\ldots,a_{k}$ satisfy (\ref{summodai}),
	 we represent every binary string $a_{i}$  and its left rotation by $\rr$ as integers:
	\begin{align*}
	a_{i}&=a_{i}^{L}2^{\wbs-\rr}+a_{i}^{R},\\
	a_{i}\lll \rr&=a_{i}^{R}2^{\rr}+a_{i}^{L}
	\end{align*}
	where $0\leq a_{i}^{L}\leq 2^{\rr}-1$ and $0\leq a_{i}^{R}\leq 2^{\wbs-\rr}-1$, are, respectively, the integers represented by the left $\rr$ places and right $\wbs-\rr$ places of the binary string $a_{i}$. Now, if 
	\begin{equation}\label{alar}
	\sum_{i=1}^{k}a_{i}^{L}=m2^{\rr}+w, \quad \sum_{i=1}^{k}a_{i}^{R}=t2^{\wbs-\rr}+s
	\end{equation}
	with $m,t,w,s$ non negative integers such that  $w\leq 2^{\rr}-1$ and $ s\leq 2^{\wbs - \rr}-1$, we have
    $$\left(\boxplus_{i=1}^{k}a_{i}\right)=\left(\left(\sum_{i=1}^{k}a_{i}^{L}\right)2^{\wbs - \rr}+\sum_{i=1}^{k}a_{i}^{R}\right)\mod 2^{\wbs}=((w+t)2^{\wbs - \rr}+s) \mod 2^{\wbs}$$
   thus, since $ s\leq 2^{\wbs - \rr}-1$ 
   \begin{equation}\label{sumrotated}
   \left(\boxplus_{i=1}^{k}a_{i}\right)\lll \rr=s 2^{\rr}+u,\quad u=(w+t)\mod 2^{\rr}
   \end{equation}
   On the other hand
	$$\boxplus_{i=1}^{k}\left(a_{i}\lll \rr\right)=\left(\left(\sum_{i=1}^{k}a_{i}^{R}\right)2^{\rr}+\sum_{i=1}^{k}a_{i}^{L}\right)\mod 2^{\wbs}=\left((s+m)2^{\rr}+w\right)\mod 2^{\wbs}	$$
and since $ w\leq 2^{\rr}-1$  
\begin{equation}\label{sumrotations}
\boxplus_{i=1}^{k}\left(a_{i}\lll \rr\right)=v2^{\rr}+w, \quad v=(s+m) \mod 2^{\wbs-\rr}
\end{equation}
Thus from (\ref{sumrotated}) and (\ref{sumrotations}) we have that  (\ref{summodai}) holds 
if and only if  
\begin{equation}\label{sumequal}
s2^{\rr}+u=v2^{\rr}+w
\end{equation}
Hence by (\ref{sumrotated}) and since $w\leq 2^{\rr}-1$ we have
$$
u=(w+t)\bmod 2^{r} =w \bmod 2^{r}
$$
which implies 
\begin{equation}\label{t}
t=c 2^{\rr} \quad  c\in \mathbb{N}
\end{equation}
and also that $u=w$.
Thus we find from equality  (\ref{sumequal}) that $s$ must be equal to $v$, which implies 
\begin{equation}\label{m}
m=d 2^{\wbs-\rr} \quad  d\in \mathbb{N}
\end{equation} 
since from (\ref{sumrotations}) we have
$v=(s+m) \bmod 2^{\wbs-\rr}$.
Therefore substituting (\ref{t}) and (\ref{m}) in (\ref{alar}) we 
observe that in order to count the number of solutions to (\ref{summodai})
we have to count the number $F(q,k,\wbs)$ of solutions in non negative integers $y_{i}$ of the systems
\begin{equation}\label{systems}
\begin{cases}
\sum_{i=1}^{k}y_{i}=h2^{\wbs}+l\\
\quad 0\leq y_{i}\leq 2^{q}-1
\end{cases}
\end{equation}
 with $ l=0,1,\ldots, 2^{q}-1$ and where  $h=0,\ldots, \left\lfloor \frac{k\left(2^{q}-1\right)}{2^{\wbs}}\right\rfloor$ since we must have
 $h2^{\wbs}+l\leq k(2^{q}-1)$ or equivalently 
 $$h+\frac{l}{2^{\wbs}}\leq\frac{k(2^{q}-1)}{2^{\wbs}}, \quad 0\leq\frac{l}{2^{\wbs}}<1$$
Thanks to Theorem 4.3 p. 138 of \cite{charalambos2002enumerative} the number of solutions of (\ref{systems}) for a fixed value of $l$ is 
$$\sum_{j=0}^{k}(-1)^{j}\binom{k}{j}
 \binom{h2^{\wbs}+l-j2^{q}+k-1}{h2^{\wbs}+l-j2^{q}},$$
thus summing  for all the values of $l$ gives
\begin{align*}
\sum_{l=0}^{2^{q}-1}\sum_{j=0}^{k}(-1)^{j}\binom{k}{j}
\binom{h2^{\wbs}+l-j2^{q}+k-1}{h2^{\wbs}+l-j2^{q}}
=
\\
=\sum_{j=0}^{k}(-1)^{j}\binom{k}{j}\sum_{i=h2^{\wbs}-j2^{q}}^{h2^{\wbs}-(j-1)2^{q}-1}\binom{i+k-1}{i}
=
\\
=
\sum_{j=0}^{k}(-1)^{j}\binom{k}{j}\left(
\binom{h2^{\wbs}-(j-1)2^{q}-1+k}{h2^{\wbs}-(j-1)2^{q}-1}-\binom{h2^{\wbs}-j2^{q}-1+k}{h2^{\wbs}-j2^{q}-1}\right)
\end{align*}
and with a final summation on the values of $h$ we obtain (\ref{fa}). 
Therefore the number of solutions to (\ref{summodai}) clearly is the product
$F(\rr,k,\wbs)F(\wbs-\rr,k,\wbs)$ and
since we have $2^{k\wbs}$ possible choices for the $k$ $\wbs$-bit strings $a_{i}$ we easily obtain (\ref{daumgen}).
\qed
\end{proof}
\begin{remark}
In Proposition \ref{sumdaumgen} we have derived a formula for the probability that equality (\ref{summodai}) holds. This result is in general different from  the one on chained modular addictions in Lemma 2 of \cite{khovratovich2015rotational} since we do not deal with a chain and so we do not request that all the conditions similar to (\ref{summodai})  involving  $a_{1},a_{2},\ldots,a_{h}$ with $h=2,\dots,k-1$ must also be simultaneously satisfied.
\end{remark}
\begin{corollary}
		Let $a, b, c$ be 
		independent and uniformly distributed strings of $\wbs$ bits, and
		$1\leq\rr\leq \wbs-1$ an integer.
		Then
		\begin{equation}
		\label{daumgen1}
		\begin{split}
		& \Pr
		\left[
		(a\lll \rr)\boxplus (b\lll \rr)\boxplus(c\lll \rr)
		=
		(a \boxplus b\boxplus c)\lll \rr
		\right]
		=\\
		&=\frac{D(2^{\rr}+2)(2^{\wbs-\rr}+2)}{9\cdot2^{\wbs}}+\mathbbm{1}_{\left\lbrace \rr=1 \vee \rr=\wbs-1  \right\rbrace} \frac{4}{2^{3\wbs}}\binom{2^{\wbs -1}}{2^{\wbs -1}-3}=P(\rr,\wbs)
		\end{split}
		\end{equation}
		where with 
		$\mathbbm{1}_{Z}$
		we indicate the usual characteristic function of $Z$, which is equal to 1 when $Z$ is true and equal to 0 when $Z$ is false.
\end{corollary}
\begin{proof}
If we use formula (\ref{daumgen}) with $k=3$ and 
$1\leq q\leq \wbs-1$ we observe that 
$$\frac{3(2^{q}-1)}{2^{\wbs}}=\frac{2^{q+1}+2^{q}-3}{2^\wbs}=\frac{2^{q+1}}{2^{\wbs}}+\frac{2^{q}-3}{2^{\wbs}}$$
thus since $0<\frac{2^{q}-3}{2^{\wbs}}<1$ we have 
$$\left\lfloor \frac{3\left(2^{q}-1\right)}{2^{\wbs}}\right\rfloor=1$$ when
$q=\wbs-1$. Therefore when we use (\ref{daumgen}) with $\rr=\wbs-1$ or, equivalently, $\rr=1$
we find from (\ref{fa}) 
$$F(1,3,\wbs)=4, \quad F(\wbs-1,3,\wbs)=\binom{2^{\wbs -1}+2}{2^{\wbs -1}-1}+\binom{2^{\wbs -1}}{2^{\wbs -1}-3}$$ 
while if $2\leq \rr \leq \wbs-2$ we have
$$F(\rr,3,\wbs)=\binom{2^{\rr}+2}{2^{\rr}-1}, \quad F(\wbs-\rr,3,\wbs)=\binom{2^{\wbs-\rr}+2}{2^{\wbs-\rr}-1}$$
since in these situations $0<\frac{3(2^{q}-1)}{2^{\wbs}}<1$  for $q=\rr, \wbs-\rr$. Thus a straightforward calculation shows that (\ref{daumgen1}) holds.
\qed
\end{proof}
\begin{remark}
We observe that when $2\leq \rr \leq \wbs-2$ this result shows a value equal to  case $k=3$ of 
Lemma 2 in \cite{khovratovich2015rotational}	, 
in which  we have
the probability of chained modular additions for
$a_{1},a_{2},\ldots,a_{k}$, 
$\wbs$-bits words chosen at random given by
\begin{equation}\label{khovra}
\Pr[\mathcal{E}] 
=
\frac{1}{2^{3\wbs}}
\binom{2^{\rr}+2}{2^{\rr}-1}
\binom{2^{\wbs-\rr}+2}{2^{\wbs-\rr}-1}=\frac{D(2^{\rr}+2)(2^{\wbs-\rr}+2)}{9\cdot2^{\wbs}} 
\end{equation}
where
\begin{equation}\label{chain}
\begin{split}
\mathcal{E}
=& 
\left[
(a_{1}\boxplus a_2)\lll \rr
=
(a_{1}\lll \rr)\boxplus(a_{2}\lll \rr)
\right]
\cap\\
&\left[
(a_{1}\boxplus a_{2}\boxplus a_{3})\lll \rr
=
(a_{1}\lll \rr)\boxplus(a_{2}\lll \rr)\boxplus(a_{3}\lll \rr)
\right]
\,.
\end{split}
\end{equation}
On the other hand when $\rr=1$ or $\rr=\wbs-1$  from (\ref{daumgen1}) we find the different value
$$
P(1,\wbs)=P(\wbs-1,\wbs)=\frac{4(2^{2\wbs-3}+1)}{3\cdot 2^{2\wbs}} 
\,.
$$
which is greater than the corresponding one given by (\ref{khovra})
This is an immediate consequence of the fact that if $k=3$ we have one more addend to be considered in  (\ref{fa}) only when $\rr=1,\wbs-1$, i. e., more solutions to (\ref{summodai}) than to the system of equalities in (\ref{chain}).
\end{remark}

We now show how to obtain an upper and lower bound for the rotational probability.
\begin{theorem}
\label{thm:quarter_round_bounds}
The rotational probability $\qrprob$ 
of a single ChaCha quarter round is such that,
\begin{equation}\label{bounds}
    D^3
    P(\rr,\wbs)
    \leq 
    \qrprob
    \leq
    \left(
    \frac{D(2^{\rr}+2)(2^{\wbs-\rr}+2)}{9\cdot2^{\wbs}}
    \right)^2
\end{equation}
\end{theorem}

\begin{proof}
Let us suppose that we can couple equations (\ref{cond1}), (\ref{cond3}) and equations (\ref{cond2}), (\ref{cond4}), considering respectively $x_{0},x_{1},b_{1}$ and $y_{3},b_{3},x_{2}$ as two triplets of random $\wbs$-bits words. Then we may find an upper bound for $\qrprob$ multiplying the probabilities of the two chains
$$\mathcal{E}_{1}= \left[(x_{0}\boxplus x_{1})\lll \rr=(x_{0}\lll \rr)\boxplus(x_{1}\lll \rr)\right]\cap$$
$$ \left[(x_{0}\boxplus x_{1}\boxplus b_{1})\lll \rr=(x_{0}\lll \rr)\boxplus(x_{1}\lll \rr)\boxplus(b_{1}\lll \rr)\right]$$
$$\mathcal{E}_{2}= \left[(b_{3}\boxplus x_{2})\lll \rr=(b_{3}\lll \rr)\boxplus(x_{2}\lll \rr)\right]\cap$$
$$\left[(y_{3}\boxplus b_{3}\boxplus x_{2})\lll \rr=(y_{3}\lll \rr)\boxplus(b_{3}\lll \rr)\boxplus(x_{2}\lll \rr)\right]$$
obtaining
$
\qrprob\leq \Pr[\mathcal{E}_{1}]\Pr[\mathcal{E}_{2}]=\Pr[\mathcal{E}]^2
\,,
$
i.e.,
$
\qrprob\leq\left(\frac{D(2^{\rr}+2)(2^{\wbs-\rr}+2)}{9\cdot2^{\wbs}}\right)^2
\,
$
since in the real situation, where those triplets of words are in general not all independent, there are less possible values which satisfy conditions (\ref{cond1}), (\ref{cond2}), (\ref{cond3}), (\ref{cond4}), with respect to all possible value that they may assume.

In order to obtain a lower bound we observe that
(\ref{cond1}) and (\ref{cond2}) hold with probability $D$ since we may consider $(x_{0},x_{1})$ and $(b_{3},x_{2})$ as couples of independent and uniformly distributed random variables. Moreover we may also request the restrictive condition that $(y_{3}, b_{3}\boxplus x_{2})$ are independent and uniformly distributed random variables such that also (\ref{cond4}) hold with probability $D$,  considering for (\ref{cond3}) the  probability given by (\ref{daumgen1}) and obtaining
$
\qrprob\geq D^3P(\rr,\wbs)
\,.
$
Thus we have
$
    D^3P(\rr,\wbs)\leq \qrprob \leq
    \left(\frac{D(2^{\rr}+2)(2^{\wbs-\rr}+2)}{9\cdot2^{\wbs}}\right)^2
    \,.
$
\qed
\end{proof}
%

\subsection{Experimental result}
\label{sec:experimental_results}

To have an additional experimental confirmation of the correctness of the bounds in \autoref{thm:quarter_round_bounds}, 
we implemented a toy version of ChaCha quarter round, 
using smaller word bit size and several different combinations of round rotations $r_0, r_1, r_2, r_3,$. 
To run the experiment, 
we exhaustively search through all possible values of 
$(x_{0},x_{1},x_{2},x_{3})$, 
then we computed 
$(\rot[\rr]{x_0}, \ldots, \rot[\rr]{x_3})$, 
evaluated both 4tuples over the quarter round function $\quarterround$,
and finally checked if 
the condition
$
(\rot[\rr]{y_{0}}, \ldots, \rot[\rr]{y_{3}})
= 
\quarterround(\rot[\rr]{x_0}, \ldots, \rot[\rr]{x_3})
$
was verified, 
and counted how many times we would happen 
(\#collisions column in \autoref{tab:qr_experimental_results}).
In \autoref{tab:qr_experimental_results}, 
we show some of the results for word size of 4, 5, and 6 bits.
The value $p$ is the probability to have a rotational collision 
for a random permutation $f$, i.e. 
$
p = 
\Pr[ 
(\rot[\rr]{y_{0}}, \ldots, \rot[\rr]{y_{3}})
= 
f(\rot[\rr]{x_0}, \ldots, \rot[\rr]{x_3})
]
$.
Notice that the case $\rr$ is equal to $\wbs-\rr$, so we do not report it in the table.

\begin{table}[ht]
    \centering
    \begin{tabular}{|c|c|c|c|c|c|}
\hline
\multicolumn{6}{|c|}{$\wbs=4, (r_0, r_1, r_2, r_3) = (1,3,2,1)$} \\
\hline
$\rr$ & \#collisions & Lower Bound & $\qrprob$ & Upper Bound & $p$ \\
\hline
1 &        747 & $0.00880 \sim 2^{-6.83}$ &  \textbf{0.01140} & $0.01373 \sim 2^{-6.19}$  &  $2^{-16.00}$ \\
2 &        388 & $0.00582 \sim 2^{-7.42}$ &  \textbf{0.00592} & $0.00954 \sim 2^{-6.71}$  &  $2^{-16.00}$ \\
\hline
\hline
\multicolumn{6}{|c|}{$\wbs = 5, (r_1, r_2, r_3, r_4) = (4,3,2,1)$} \\
\hline
$\rr$ & \#collisions & Lower Bound & $\qrprob$ & Upper Bound & $p$ \\
\hline
1 &       8917 & $0.00630 \sim 2^{-7.31}$ &  \textbf{0.00850} & $0.00992 \sim 2^{-6.66}$  &  $2^{-20.00}$ \\
2 &       3405 & $0.00318 \sim 2^{-8.30}$ &  \textbf{0.00325} & $0.00536 \sim 2^{-7.54}$  &  $2^{-20.00}$ \\
\hline
\hline
\multicolumn{6}{|c|}{$\wbs=6, (r_1, r_2, r_3, r_4) = (5,3,2,1)$} \\
\hline
$\rr$ & \#collisions & Lower Bound & $\qrprob$ & Upper Bound & $p$ \\
\hline
1 &     123317 & $0.00528 \sim 2^{-7.57}$ &  \textbf{0.00735} & $0.00834 \sim 2^{-6.91}$  &  $2^{-24.00}$ \\
2 &      39482 & $0.00228 \sim 2^{-8.78}$ &  \textbf{0.00235} & $0.00388 \sim 2^{-8.01}$  &  $2^{-24.00}$ \\
3 &      32628 & $0.00174 \sim 2^{-9.17}$ &  \textbf{0.00194} & $0.00302 \sim 2^{-8.37}$  &  $2^{-24.00}$ \\
\hline
    \end{tabular}
    \caption{Experimental results on a toy version of ChaCha quarter round.}
    \label{tab:qr_experimental_results}
\end{table}

\subsection{Bounds propagation through the full round}
\label{sec:bounds_one_round}

We indicate with 
\begin{itemize}
    \item $Y = \round(X)$ the application of one round of the ChaCha permutation (either a column or a diagonal round).
    \item $Y = \round^i(X)$ the application of $i$ consecutive round of the ChaCha permutation, alternating column to diagonal rounds (where the first round $\round^1$ is a column round).
\end{itemize}

The following theorem shows how to extend the lower and  upper bounds of \autoref{thm:quarter_round_bounds} from the ChaCha quarter round to one full round of the ChaCha permutation.
\begin{theorem}
\label{thm:bounds_one_round}
Let
$\lb, \ub$ be such that
$ 
\lb
\leq 
\Pr\left[\prot{{\bm{y}}} = \quarterround(\prot{\bm{x}}) \right]
\leq
\ub
$.
Then
\begin{equation}
\lb^\std
\leq 
\Pr\left[\prot{{Y}} = \round(\prot{X}) \right]
\leq
\ub^\std
\end{equation} 
\end{theorem}
\begin{proof}
Since a full round applies 
$\std$ quarter rounds independently in parallel,
to extend the bounds from \autoref{thm:quarter_round_bounds} 
it is sufficient to multiply the probabilities, i.e.,
for the rounds where the quarter round is applied to the columns we have
\begin{align*}
& 
\Pr
    \left[
        \prot{{Y}} 
        = 
        \round(\prot{X}) 
    \right]
= \\
& 
\Pr
\left[
    \left(
        \begin{matrix}
            \rot{y_{0,0}} \\
            \vdots  \\
            \rot{y_{\std-1,0}}  
        \end{matrix}
    \right)
    =
    \quarterround
    \left(
        \begin{matrix}
            \rot{x_{0,0}} \\
            \vdots \\
            \rot{y_{\std-1,0}} 
        \end{matrix}
    \right)
    \wedge
    \cdots
    \wedge
        \left(
        \begin{matrix}
            \rot{y_{0,\std-1}} \\
            \vdots  \\
            \rot{y_{\std-1,\std-1}}  
        \end{matrix}
    \right)
    =
    \quarterround
    \left(
        \begin{matrix}
            \rot{x_{0,\std-1}} \\
            \vdots \\
            \rot{y_{\std-1,\std-1}} 
        \end{matrix}
    \right)
\right]
= \\
&
\Pr
\left[
    \left(
        \begin{matrix}
            \rot{y_{0,0}} \\
            \vdots  \\
            \rot{y_{\std-1,0}}  
        \end{matrix}
    \right)
    =
    \quarterround
    \left(
        \begin{matrix}
            \rot{x_{0,0}} \\
            \vdots \\
            \rot{y_{\std-1,0}} 
        \end{matrix}
    \right)
\right]
\cdot
\ldots
\cdot
\Pr
\left[
    \left(
        \begin{matrix}
            \rot{y_{0,\std-1}} \\
            \vdots  \\
            \rot{y_{\std-1,\std-1}}  
        \end{matrix}
    \right)
    =
    \quarterround
    \left(
        \begin{matrix}
            \rot{x_{0,\std-1}} \\
            \vdots \\
            \rot{y_{\std-1,\std-1}} 
        \end{matrix}
    \right)
\right]
\,.
\end{align*}
For the rounds where the quarter round is applied to the diagonals, 
the proof is alike.
\qed
\end{proof}
Recall that, 
in \autoref{thm:quarter_round_bounds}, 
for $\std = 4$, 
we proved that
$
\lb =
D^3P(\rr,\wbs)
$ 
and 
$
\ub =
\left(
\frac{D(2^{\rr}+2)(2^{\wbs-\rr}+2)}{9\cdot2^{\wbs}}
\right)^2
$. 

\subsection{Bounds propagation through the full permutation}
\label{sec:bounds_full_permutation}

The following theorem shows how to extend the lower and  upper bounds of \autoref{thm:bounds_one_round} from one round of ChaCha to $i$ consecutive rounds.
To prove the theorem, 
we make an assumption that seems to be a good approximation of what happens in practice, i.e.
we assume that the input states of each round 
are independent and uniformly distributed.

\begin{theorem}
\label{thm:bounds_many_rounds}
Let
$\lb, \ub$ be such that
$ 
\lb
\leq 
\Pr\left[\prot{{\bm{y}}} = \quarterround(\prot{\bm{x}}) \right]
\leq
\ub
$.
Then
\begin{equation}
\lb^{\std i}
\leq 
\Pr\left[\prot{{Y}} = \round^i(\prot{X}) \right]
\leq
\ub^{\std i}
\end{equation} 
\end{theorem}

\begin{proof}
Because of \autoref{thm:bounds_one_round}, 
we have that
$
\lb^\std
\leq 
\Pr\left[\prot{{Y}} = \round^1(\prot{X}) \right]
\leq
\ub^\std
$.
For the inductive step, notice that
$
\Pr
\left[
\prot{{Y}} = \round^i(\prot{X}) 
\right] 
=
\Pr
\left[
\prot{{Y}} = \round(\round^{i-1}(\prot{X})) 
\right] 
$.
Thus, for the assumption of independence of each state, 
we have that the equality
$
\Pr
\left[
\prot{{Y}} = \round(\round^{i-1}(\prot{X})) 
\right] 
=
\Pr
\left[
\prot{{Y}} = \round(\prot{\round^{i-1}(X)}) 
\right] 
$
holds with probability bounded by $\lb^\std$ and $\ub^\std$.
\qed
\end{proof}

\section{Distinguisher description}
\label{sec:distinguisher}

When $\prot[\rr]{{\bm{y}}} = F(\prot[\rr]{\bm{x}})$, 
we say that $F$ has a 
\emph{parallel rotational collision} 
(or simply a \emph{rotational collision}) 
in $\bm{x}$ 
with respect to $\rr$. 
In this section, we show that, 
up to a certain number of rounds, 
ChaCha permutation has more rotational collisions 
with respect to a random permutation with a fixed point.
We first describe what is the probability to have a rotational collision for a random permutation $\Pi$ with a fixed point.
Then, we use this probability and the bounds from 
\autoref{sec:bounds_full_permutation} 
to distinguish ChaCha permutation from $\Pi$.

\subsection{Rotational collisions of a random permutation}
\label{sec:collisions_random_permutation}

For every set $A$, let $\mathcal{S}(A)$ 
be the group of permutations over $A$.
Moreover, for each permutation 
${\Pi} : (\mathbb{F}_2^\wbs)^k \to (\mathbb{F}_2^\wbs)^k$ let
$
C_{\Pi} 
:= 
\#
\big\{
\bm{x} \in (\mathbb{F}_2^\wbs)^k : {\Pi}(\prot{\bm{x}}) = \prot{{\Pi}(\bm{x})}
\big\} 
$
be the number of rotational collisions of ${\Pi}$.
We want first to compute the expected number of rotational collisions of a random permutation.

\begin{proposition}
\label{lem:fixed}
We have
$
\#
\big\{
\bm{x} \in \big(\mathbb{F}_2^\wbs\big)^k : \bm{x} = \prot{\bm{x}}
\big\} 
= 
2^{k\gcd(\wbs, \rr)}
\,.
$
\end{proposition}
\begin{proof}
For each 
$\bm{x} = (x_1, \dots, x_k) 
\in 
\big(\mathbb{F}_2^\wbs\big)^k$ 
we have 
$\bm{x} = \prot{\bm{x}}$ 
if and only if 
$x_i = \rot{x_i}$ 
for each 
$i \in \{1,\dots,k\}$.
Hence, it is enough to show that
$
\#
\big\{
x \in \mathbb{F}_2^\wbs : x = \rot{x}
\big\} 
= 
2^{\gcd(\wbs, \rr)}
\,.
$
In turn, this is equivalent to the assertion that the permutation of $\mathbb{Z}/\wbs\mathbb{Z}$ 
given by $k \mapsto k + \rr$ 
has 
$\gcd(\wbs, \rr)$ cycles, 
which is a well-known fact.
\qed
\end{proof}
We can now compute the expected number of rotational collisions of a random permutation.
\begin{proposition}
Let ${\Pi}$ be a uniformly random variable in $\mathcal{S}\big((\mathbb{F}_2^\wbs)^k\big)$.
Then
\begin{equation*}
\mathbb{E}[C_{\Pi}] 
= 
\frac{2^{\wbs k} + 2^{2k \gcd(\wbs, \rr)} - 2^{k \gcd(\wbs, \rr) + 1}}{2^{\wbs k} - 1} .
\end{equation*}
\end{proposition}
\begin{proof}
By the definition of expected value, we have
\begin{align*}
\mathbb{E}[C_{\Pi}] &= \frac1{\# \mathcal{S}\big((\mathbb{F}_2^\wbs)^k\big)} \sum_{{\Pi} \in \mathcal{S}((\mathbb{F}_2^\wbs)^k)} \#\big\{\bm{x} \in (\mathbb{F}_2^\wbs)^k : {\Pi}(\prot{\bm{x}}) = \prot{{\Pi}(\bm{x})}\big\} \\
&= \frac1{(2^{\wbs k})!} \sum_{{\Pi} \in \mathcal{S}((\mathbb{F}_2^\wbs)^k)} \sum_{\bm{x} \in (\mathbb{F}_2^\wbs)^k} \mathbbm{1} \big[{\Pi}(\prot{\bm{x}}) = \prot{{\Pi}(\bm{x})}\big] \\
&= \frac1{(2^{\wbs k})!} \sum_{\bm{x} \in (\mathbb{F}_2^\wbs)^k} \sum_{{\Pi} \in \mathcal{S}((\mathbb{F}_2^\wbs)^k)} \mathbbm{1} \big[{\Pi}(\prot{\bm{x}}) = \prot{{\Pi}(\bm{x})}\big] \\
&= \frac1{(2^{\wbs k})!} \sum_{\bm{x} \in (\mathbb{F}_2^\wbs)^k} \#\big\{{\Pi} \in \mathcal{S}\big((\mathbb{F}_2^\wbs)^k\big) : {\Pi}(\prot{\bm{x}}) = \prot{{\Pi}(\bm{x})}\big\} \\
&= \frac1{(2^{\wbs k})!} \sum_{\bm{x}, \bm{y} \in (\mathbb{F}_2^\wbs)^k} N_{\bm{x}, \bm{y}} ,
\end{align*}
where
$
N_{\bm{x}, \bm{y}} 
:= 
\#
\big\{
{\Pi} \in \mathcal{S}\big((\mathbb{F}_2^\wbs)^k\big) : {\Pi}(\bm{x}) = \bm{y} \land {\Pi}(\prot{\bm{x}}) = \prot{\bm{y}} 
\big\}
\,,
$
for every $\bm{x}, \bm{y} \in (\mathbb{F}_2^\wbs)^k$.
Hence, we have to compute $N_{\bm{x}, \bm{y}}$.
There are four cases:
\begin{enumerate}
\item If $\bm{x} = \prot{\bm{x}}$ and $\bm{y} = \prot{\bm{y}}$, 
      then $N_{\bm{x}, \bm{y}} = (2^{\wbs k} - 1)!$
\item If $\bm{x} \neq \prot{\bm{x}}$ and $\bm{y} \neq \prot{\bm{y}}$, 
      then $N_{\bm{x}, \bm{y}} = (2^{\wbs k} - 2)!$
\item If $\bm{x} \neq \prot{\bm{x}}$ and $\bm{y} = \prot{\bm{y}}$, 
      then $N_{\bm{x}, \bm{y}} = 0$
\item If $\bm{x} = \prot{\bm{x}}$ and $\bm{y} \neq \prot{\bm{y}}$,    
      then $N_{\bm{x}, \bm{y}} = 0$
\end{enumerate}
Consequently, using also \autoref{lem:fixed}, 
we get the claimed formula:
\begin{align*}
\mathbb{E}[C_{\Pi}] &= \frac1{(2^{\wbs k})!} \left(\sum_{\substack{\bm{x}, \bm{y} \in (\mathbb{F}_2^\wbs)^k 
\\ 
\bm{x} = \prot{\bm{x}},\; \bm{y} = \prot{\bm{y}}}} (2^{\wbs k} - 1)! + \sum_{\substack{\bm{x}, \bm{y} \in (\mathbb{F}_2^\wbs)^k 
\\ 
\bm{x} \neq \prot{\bm{x}},\; \bm{y} \neq \prot{\bm{y}}}} (2^{\wbs k} - 2)!\right) \\
&= \frac1{2^{\wbs k}} \sum_{\substack{\bm{x}, \bm{y} \in (\mathbb{F}_2^\wbs)^k \\ \bm{x} = \prot{\bm{x}},\; \bm{y} = \prot{\bm{y}}}} 1 + \frac1{2^{\wbs k}(2^{\wbs k}-1)}\sum_{\substack{\bm{x}, \bm{y} \in (\mathbb{F}_2^\wbs)^k \\ \bm{x} \neq \prot{\bm{x}},\; \bm{y} \neq \prot{\bm{y}}}} 1 \\
&= \frac1{2^{\wbs k}} \cdot 2^{2k \gcd(\wbs, \rr)} + \frac1{2^{\wbs k}(2^{\wbs k}-1)} \cdot \big(2^{\wbs k} - 2^{k\gcd(\wbs, k)}\big)^2 \\
&= \frac{2^{\wbs k} + 2^{2k \gcd(\wbs, \rr)} - 2^{k \gcd(\wbs, \rr) + 1}}{2^{\wbs k} - 1} 
\,.
\end{align*}
\qed
\end{proof}

For $\wbs=32$, $\std=4$, and $\rr=1$, then 
$\mathbb{E}[C_{\Pi}]$ is basically 1.
As a consequence, for a random permutation $\Pi$ with a fixed point, 
then $\mathbb{E}[C_{\Pi}]$ is basically 2.

\subsection{ChaCha permutation vs random permutation}
\label{sec:chacha_vs_random}

In \autoref{tab:bound_propagation}, 
we display the lower and upper bounds 
of \autoref{sec:bounds_full_permutation}, 
\autoref{thm:bounds_many_rounds}, 
for $\wbs=32$, $\std=4$, $\rr=1$, and rounds from 1 to 20.
As we showed in \autoref{sec:collisions_random_permutation}, 
for a random permutation 
$\Pi \in \mathcal{S}(\FF_2^{\std\wbs})$ with one fixed point, 
a rotational collision happens with probability
very close to $2/2^{\std\wbs}$.
In the case of ChaCha parameters this probability is $1/2^{511}$.
Thus, we can build a distinguisher $\mathcal{A}$ 
with access to an oracle $\oracle$
running either $\chachaperm$ with $\rho$ rounds or $\Pi$.
Let $\lb^{\std i}$ and $\ub^{\std i}$ be, respectively, 
the upper and lower bound of \autoref{thm:bounds_many_rounds}, 
with $i=1,\ldots,\rho$. 
The algorithm $\mathcal{A}$ runs as follow:
generate binary strings 
$X_i \in \FF_2^{\std\wbs}$ 
for 
$i = 1, \ldots, 
\lceil 1/\ub^{\std i} \rceil$;
ask the oracle the corresponding output 
$Y_i = \oracle(X_i)$;
if there exists $i$ such that 
$\prot{Y_i} = \oracle(\prot{X_i})$  
then the algorithm says the oracle is running $\chachaperm$.
If such $i$ does not exists, 
then the oracle is running $\Pi$.

The complexity of the algorithm $\mathcal{A}$ 
is dominated by the $2 \lceil 1 / \ub^{\std i} \rceil$ 
calls to the oracle. 
For example, to distinguish $\chachaperm$ with 8 rounds, 
$\mathcal{A}$ performs $2^{231}$ calls to the oracle, 
while for $\chachaperm$ with 17 rounds, 
the calls are $2^{489}$.
After the 17th round, $\mathcal{A}$ can not distinguish $\chachaperm$ from $\Pi$ anymore.

\begin{table}[ht]
    \centering
    \resizebox{\textwidth}{!}{
    \begin{tabular}{c|l|l||c|l|l}
Round & Lower Bound $\lb^{\std i}$ & Upper Bound $\ub^{\std i}$ & 
Round & Lower Bound $\lb^{\std i}$ & Upper Bound $\ub^{\std i}$
\\
\hline
  1 & $\sim 2^{-27.32}$  & $\sim 2^{-28.68}$  & 11 & $\sim 2^{-300.52}$ & $\sim 2^{-315.48}$ \\
  2 & $\sim 2^{-54.64}$  & $\sim 2^{-57.36}$  & 12 & $\sim 2^{-327.84}$ & $\sim 2^{-344.16}$ \\
  3 & $\sim 2^{-81.96}$  & $\sim 2^{-86.04}$  & 13 & $\sim 2^{-355.16}$ & $\sim 2^{-372.84}$ \\
  4 & $\sim 2^{-109.28}$ & $\sim 2^{-114.72}$ & 14 & $\sim 2^{-382.48}$ & $\sim 2^{-401.52}$ \\
  5 & $\sim 2^{-136.60}$ & $\sim 2^{-143.40}$ & 15 & $\sim 2^{-409.80}$ & $\sim 2^{-430.20}$ \\
  6 & $\sim 2^{-163.92}$ & $\sim 2^{-172.08}$ & 16 & $\sim 2^{-437.12}$ & $\sim 2^{-458.88}$ \\
  7 & $\sim 2^{-191.24}$ & $\sim 2^{-200.76}$ & 17 & $\sim 2^{-464.45}$ & $\sim 2^{-487.55}$ \\
  8 & $\sim 2^{-218.56}$ & $\sim 2^{-229.44}$ & 18 & $\sim 2^{-491.77}$ & $\sim 2^{-516.23}$ \\
  9 & $\sim 2^{-245.88}$ & $\sim 2^{-258.12}$ & 19 & $\sim 2^{-519.09}$ & $\sim 2^{-544.91}$ \\
 10 & $\sim 2^{-273.20}$ & $\sim 2^{-286.80}$ & 20 & $\sim 2^{-546.41}$ & $\sim 2^{-573.59}$ \\
    \end{tabular}
    }
    \caption{Bounds propagation through ChaCha rounds with $\wbs=32$, $\std=4$, and $\rr=1$.}
    \label{tab:bound_propagation}
\end{table}

\section{Conclusion}
\label{sec:conclusion}

We showed that parallel rotational collisions are more likely to happen in ChaCha underlying permutation with up to 17 rounds, than in a random permutation of the same input size. We are not aware of any theoretical study of ChaCha rotational properties, and we leave to future research finding an application of our results to the cryptanalysis of ChaCha stream cipher.


\bibliographystyle{splncs04}
\bibliography{biblio}

\end{document}